\documentclass[cleveref,a4paper,UKenglish]{lipics-v2021}
\usepackage{cite} % sort cite numbers
\usepackage{amsmath,amsthm}
\usepackage{amssymb}  
\usepackage{xspace}
\usepackage{adjustbox,standalone}

\usepackage{tikz}
\tikzstyle{every picture} = [>=latex]
\usetikzlibrary{calc,arrows,decorations.markings,quotes}
\usetikzlibrary{backgrounds,shapes,arrows}

 \usepackage[appendix=inline]{apxproof}
\ifx\proof\inlineproof
  \def\apxmark{}
\else
  \def\apxmark{\,*$\!$}
\fi
\newtheoremrep{propositionx2}[theorem]{Proposition}
\newtheoremrep{proposition2}[theorem]{Proposition\apxmark}
\newtheoremrep{theorem2}[theorem]{Theorem\apxmark}
\newtheoremrep{lemma2}[theorem]{Lemma\apxmark}
\newtheoremrep{claim2}[theorem]{Claim\apxmark}
\def\ca#1{{\cal#1}}

\def\sn{\mathop{\text{sn}}}
\def\qn{\mathop{\text{qn}}}

\hideLIPIcs
\nolinenumbers

\title{Stack and Queue Numbers of Graphs Revisited}
\author{Petr Hlin{\v e}n\'y}{Masaryk University, Brno,
  Czech republic}{hlineny@fi.muni.cz}{https://orcid.org/0000-0003-2125-1514}{}
\author{Adam Straka}{Masaryk University, Brno,
  Czech republic}{493034@mail.muni.cz}{}{}
  
\authorrunning{P.\ Hlin\v{e}n\'y and A.~Straka}
\Copyright{Petr Hlin\v{e}n\'y}


\begin{document}
\maketitle

\begin{abstract}
A long-standing question of the mutual relation between the stack and queue numbers of a graph,
explicitly emphasized by Dujmovi\'c and Wood in 2005, was partially answered
by Dujmovi\'c, Eppstein, Hickingbotham, Morin and Wood in 2022;
they proved the existence of a graph family with the queue number at most $4$ but unbounded stack number.
We give an alternative very short, and still elementary, proof of the same fact.
\end{abstract}

\section{Introduction}
%%%%%%%%%%%%%%%%%%%%%%%%%%%%%%%%%%%%%%%%%%%%%%%%%%%%%%%%%%%%%%%%%%%%%%%

The graph invariants stack number and queue number are related to the concept of graph linearization.  
It consists of ``stretching a graph onto a line' (i.e., giving a linear order to the vertex set) and processing the graph (including its edges) in this order.
During this processing, the edges can be thought as stored in various standard data structures, such as stacks or queues.  
More precisely, when the processing encounters an end vertex of an edge $f$ for the first time,
it pushes $f$ onto a chosen stack (or queue), and once it reaches the other end vertex of $f$, it pops $f$ from its stack (or queue).
Of course, the key property of the stack (``last in, first out'') or of the queue (``first in, first out'') can never be broken,
and this restriction typically forces us to use more than one of the considered data structures to store different edges of our graph.
The goal is to find a vertex order of the graph and a distribution of its edges which minimize the number of stacks (or queues) used.

% The edge processing then looks as follows.  We go along the *************
% ordered vertices.  When we encounter left end of an edge, we push it on it's
% designated stack (or queue).  When we reach the right end of an edge we pop
% the edge from its stack (queue).  We can never break the stack (queue)
% property during the processing.  The invariants describe how well their
% corresponding data structures perform in graph linearization (i.e.  how many
% stack/queues are needed to linearize a graph).

To determine which of the two data structures is better, we would like to know whether the stack number is bounded by the queue number or vice versa.  
If the queue number is bounded by the stack number and the stack number is not bounded by the queue number,
then queues are strictly better than stacks.  
Stacks would be better in the opposite situation.
The question of mutual bounds between the stack and queue numbers was emphasized by Dujmovi\'c and Wood in~\cite{dujmovic2005stacks}.

We give the formal definitions first.
Consider a graph $G$ and a strict linear order $\prec$ on its vertex set $V(G)$.
Two edges $xx',yy'\in E(G)$ with $x\prec x'$ and $y\prec y'$ are said to {\em$\prec$-cross} if $x\prec y\prec x'\prec y'$ or $y\prec x\prec y'\prec x'$,
and to  {\em$\prec$-nest} if $x\prec y\prec y'\prec x'$ or $y\prec x\prec x'\prec y'$.
See \Cref{fig:crossnest}. A set of $k$ pairwise crossing (nested) edges is called a {\em$k$-twist} (a {\em$k$-rainbow}). 
The {\em stack number} $\sn(G)$ ({\em queue number} $\qn(G)$) of a graph $G$ is the minimum integer $k$ such that there exist a linear order $\prec$ of $V(G)$
and a colouring of the edges of $G$ by $k$ colours such that no two edges of the same colour $\prec$-cross ($\prec$-nest, resp.). Having no two edges of the same colour crossed (nested) is equivalent to not breaking the stack (queue) property while processing the edges.
The corresponding order $\prec$ together with the colouring is called a {\em$k$-stack ($k$-queue, resp.) layout} of~$G$.

\begin{figure}[ht]
\centering
~\raise2ex\hbox{a)~}
\begin{tikzpicture}
\tikzstyle{every node}=[draw, shape=circle, minimum size=3pt,inner sep=0pt, fill=black]
\draw[dotted] (0,0)--(5,0);
\node[label=below:$\prec:$,draw=none,fill=none] at (0,0) (0) {};
\node[label=below:$~x\,~$] at (1,0) (x) {}; \node[label=below:$x'$] at (3,0) (xx) {};
\node[label=below:$~\,y\,~$] at (2,0) (y) {}; \node[label=below:$y'$] at (4,0) (yy) {};
\draw (x) to[bend left] (xx) (y) to[bend left] (yy);
\end{tikzpicture}
\qquad
\raise2ex\hbox{b)~}
\begin{tikzpicture}
\tikzstyle{every node}=[draw, shape=circle, minimum size=3pt,inner sep=0pt, fill=black]
\draw[dotted] (0,0)--(5,0);
\node[label=below:$\prec:$,draw=none,fill=none] at (0,0) (0) {};
\node[label=below:$~x\,~$] at (1,0) (x) {}; \node[label=below:$x'$] at (4,0) (xx) {};
\node[label=below:$~\,y\,~$] at (2,0) (y) {}; \node[label=below:$y'$] at (3,0) (yy) {};
\draw (x) to[bend left] (xx) (y) to[bend left] (yy);
\end{tikzpicture}
\caption{Edges $xx'$ and $yy'$ that (a) $\prec$-cross, and (b) $\prec$-nest.}
\label{fig:crossnest}
\end{figure}
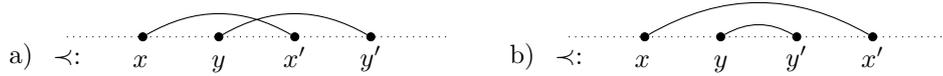

\vspace*{-2ex}\subparagraph{Stack Number.}
The stack number was originally defined in the sixties as the book thickness (or page number) by Atneosen and Persinger 
in \cite{atneosen1968embeddability,persinger1966subsets}.  
The book embedding problem consists of laying out the vertices of a graph onto a line (the back of the book) and drawing its
edges onto half-planes (pages) incident to this line, such that no two edges cross.  
% The pages share the back of the book as a common boundary.  
In the original definition, it was allowed to draw edges on the back of the book itself,
which allows to embed every linear forest in $0$ pages.
Aside of this possible degeneracy in the case of linear forests, the book thickness of a graph is equal to its stack number.

Stack number has been extensively researched and has found many applications.  
Some of them include creating integrated circuits via VLSI design \cite{vlsi}, direct interconnection networks~\cite{KAPOOR2002267} or
fault-tolerant processor arrays~\cite{rosenberg1983diogenes}.

It is $\mathbf{NP}$-hard to determine whether the stack number of a graphs is $k$ for every fixed $k \geq 4$, even for a fixed ordering \cite{vlsi}.
If we layout the vertices in this order on a circle, we can see that the problem becomes the same as deciding $k$-colourability of 
a circle graph (which is the intersection graph of chords of a circle).  
Interestingly, $3$-colourability of circle graphs has recently been re-declared open in~\cite{DBLP:journals/corr/abs-2309-02258}.
The $1$-stack graphs are known to be exactly outerplanar graphs and $2$-stack graphs are subgraphs of planar Hamiltonian graphs \cite{bernhart1979book}.  
There are many other graph classes known with bounded stack number, such as trees, $X$-trees or planar~graphs~\cite{vlsi, yannakakis1986four}.

\vspace*{-2ex}\subparagraph{Queue Number.}
The queue number, on the other hand, was defined in 1992 by Heath and Rosenberg in \cite{laying_queues}.  
Its applications include fault tolerant processor arrays \cite{rosenberg1983diogenes} and permuting object using queues \cite{tarjan1972sorting}.
It is \textbf{NP}-hard to determine the queue number of a graph \cite{laying_queues}, even for the solution value~$1$
(the $1$-queue graphs are exactly arched leveled planar graphs \cite{laying_queues}).
However we can find the queue number of a graph in polynomial time for a fixed ordering, unlike for the stack number; the queue number
of a given ordering is equal to the size of the largest rainbow in the layout \cite{laying_queues}.  
Other classes of bounded queue number include trees, $X$-trees, unicyclic graphs or planar graphs \cite{laying_queues, dujmovic2020planar}.

\vspace*{-2ex}\subparagraph{Mutual relations.}
There are several known results concerning relations between the stack and queue numbers. 
Every $1$-stack graph accepts a $2$-queue layout \cite{comparing}, and every $1$-queue graph accepts a $2$-stack layout \cite{laying_queues}. 
$2$-stack graphs are planar and therefore have bounded queue number as well~\cite{dujmovic2020planar}. 
It is unknown whether $3$-stack graphs have bounded queue number, but this question is already equivalent to the question whether the queue number 
is bounded by the stack number \cite{dujmovic2005stacks}. 
More specifically, the queue number is bounded by the stack number if and only if the bipartite $3$-stack graphs have bounded queue number. 
However, as every graph has a $3$-stack subdivision \cite{blankenship1999drawing}, and we can create such a subdivision
with odd numbers of dividing vertices, thus creating a bipartite graph,
bounding the queue number of bipartite $3$-stack graphs will likely not be a simple task.  
% The best known upper bound on the number of division vertices is $\mathcal{O}(log(max\{sn(G), qn(G)\}))$ \cite{dujmovic2005stacks}.

In the other direction, quite recently in 2022, Dujmovi\'c, Eppstein,
Hickingbotham, Morin and Wood \cite{dujmovic2021stack} constructed a class of graphs with the queue number at most $4$ and unbounded stack number.  
This shows that the stack number is not bounded by the queue number and therefore stacks are {\em not more powerful} than queues.

To state the main result of \cite{dujmovic2021stack}, we define the following special graph $H_n$:
the vertex set is $V(H_n)=\{1,\ldots,n\}^2$, and $uv\in E(H_n)$ where $u=[a,b]\in V(H_n)$ and $v=[c,d]\in V(H_n)$, if and only if
$|a-c|+|b-d|=1$ or $a-c=b-d\in\{-1,1\}$.
Note that $H_n$ is the plane dual of the hexagonal (``honeycomb'') grid, and see an illustration in \Cref{fig:SnHn}.

\begin{figure}[ht]
    \centering%\vspace*{-3.1ex}
~\raise2ex\hbox{a)~}
\begin{tikzpicture}[scale=0.6, every node/.style={circle,fill,inner sep=0pt, minimum size=0.12cm}]
    \node (S) at (2,3) { };
    \node (A) at (0,0) { };
    \node (B) at (1,0) { };
    \node (C) at (2,0) { };
    \node (D) at (3,0) { };
    \node (E) at (4,0) { };    
    \path (S) edge (A); 
    \path (S) edge (B); 
    \path (S) edge (C);
    \path (S) edge (D);
    \path (S) edge (E);
\end{tikzpicture}
\qquad\raise2ex\hbox{b)~}
\begin{tikzpicture}[scale=0.75, every node/.style={circle,fill,inner sep=0pt, minimum size=0.12cm}]   
    \node (1) at (0,0) { };
    \node (2) at (1.5,0) { };
    \node (3) at (3,0) { };
    \node (4) at (0,1.5) { };
    \node (5) at (1.5,1.5) { };
    \node (6) at (3,1.5) { }; 
    \node (7) at (0,3) { };
    \node (8) at (1.5,3) { };
    \node (9) at (3,3) { };
    
    \path (1) edge (2); 
    \path (2) edge (3); 
    \path (4) edge (5); 
    \path (5) edge (6); 
    \path (7) edge (8); 
    \path (8) edge (9); 
    \path (1) edge (5); 
    \path (2) edge (6); 
    \path (4) edge (8); 
    \path (5) edge (9); 
    \path (1) edge (4); 
    \path (4) edge (7);
    \path (2) edge (5); 
    \path (5) edge (8); 
    \path (3) edge (6); 
    \path (6) edge (9); 
\end{tikzpicture}
\qquad\raise2ex\hbox{c)~}
\begin{tikzpicture}[scale=0.33, every node/.style={circle,fill=black, inner sep=0.4mm}]
    \node (S1) at (2,3) { };    \node (A1) at (0,0) { };    \node (B1) at (1,0) { };    \node (C1) at (2,0) { };    \node (D1) at (3,0) { };    \node (E1) at (4,0) { }; 
    \path (S1) edge (A1);     \path (S1) edge (B1);     \path (S1) edge (C1);    \path (S1) edge (D1);    \path (S1) edge (E1);
    \node (S2) at (8,3) { };    \node (A2) at (6,0) { };    \node (B2) at (7,0) { };    \node (C2) at (8,0) { };    \node (D2) at (9,0) { };    \node (E2) at (10,0) { }; 
    \path (S2) edge (A2);     \path (S2) edge (B2);     \path (S2) edge (C2);    \path (S2) edge (D2);    \path (S2) edge (E2);
    \node (S3) at (14,3) { };    \node (A3) at (12,0) { };    \node (B3) at (13,0) { };    \node (C3) at (14,0) { };    \node (D3) at (15,0) { };    \node (E3) at (16,0) { }; 
    \path (S3) edge (A3);     \path (S3) edge (B3);     \path (S3) edge (C3);    \path (S3) edge (D3);    \path (S3) edge (E3);
    \node (S4) at (2,7) { };    \node (A4) at (0,4) { };    \node (B4) at (1,4) { };    \node (C4) at (2,4) { };    \node (D4) at (3,4) { };    \node (E4) at (4,4) { }; 
    \path (S4) edge (A4);     \path (S4) edge (B4);     \path (S4) edge (C4);    \path (S4) edge (D4);    \path (S4) edge (E4);
    \node (S5) at (8,7) { };    \node (A5) at (6,4) { };    \node (B5) at (7,4) { };    \node (C5) at (8,4) { };    \node (D5) at (9,4) { };    \node (E5) at (10,4) { }; 
    \path (S5) edge (A5);     \path (S5) edge (B5);     \path (S5) edge (C5);    \path (S5) edge (D5);    \path (S5) edge (E5);
    \node (S6) at (14,7) { };    \node (A6) at (12,4) { };    \node (B6) at (13,4) { };    \node (C6) at (14,4) { };    \node (D6) at (15,4) { };    \node (E6) at (16,4) { }; 
    \path (S6) edge (A6);     \path (S6) edge (B6);     \path (S6) edge (C6);    \path (S6) edge (D6);    \path (S6) edge (E6);
    \node (S7) at (2,11) { };    \node (A7) at (0,8) { };    \node (B7) at (1,8) { };    \node (C7) at (2,8) { };    \node (D7) at (3,8) { };    \node (E7) at (4,8) { }; 
    \path (S7) edge (A7);     \path (S7) edge (B7);     \path (S7) edge (C7);    \path (S7) edge (D7);    \path (S7) edge (E7);
    \node (S8) at (8,11) { };    \node (A8) at (6,8) { };    \node (B8) at (7,8) { };    \node (C8) at (8,8) { };    \node (D8) at (9,8) { };    \node (E8) at (10,8) { }; 
    \path (S8) edge (A8);     \path (S8) edge (B8);     \path (S8) edge (C8);    \path (S8) edge (D8);    \path (S8) edge (E8);
    \node (S9) at (14,11) { };    \node (A9) at (12,8) { };    \node (B9) at (13,8) { };    \node (C9) at (14,8) { };    \node (D9) at (15,8) { };    \node (E9) at (16,8) { }; 
    \path (S9) edge (A9);     \path (S9) edge (B9);     \path (S9) edge (C9);    \path (S9) edge (D9);    \path (S9) edge (E9);
    \path (S1) edge[color=blue, opacity=0.4, bend left=20] (S2); 
    \path (S2) edge[color=blue, opacity=0.4, bend left=20] (S3); 
    \path (S4) edge[color=blue, opacity=0.4, bend left=20] (S5); 
    \path (S5) edge[color=blue, opacity=0.4, bend left=20] (S6); 
    \path (S7) edge[color=blue, opacity=0.4, bend left=20] (S8); 
    \path (S8) edge[color=blue, opacity=0.4, bend left=20] (S9); 
    \path (A1) edge[color=blue, opacity=0.4, bend left=20] (A2); 
    \path (A2) edge[color=blue, opacity=0.4, bend left=20] (A3); 
    \path (A4) edge[color=blue, opacity=0.4, bend left=20] (A5); 
    \path (A5) edge[color=blue, opacity=0.4, bend left=20] (A6); 
    \path (A7) edge[color=blue, opacity=0.4, bend left=20] (A8); 
    \path (A8) edge[color=blue, opacity=0.4, bend left=20] (A9); 
    \path (B1) edge[color=blue, opacity=0.4, bend left=20] (B2); 
    \path (B2) edge[color=blue, opacity=0.4, bend left=20] (B3); 
    \path (B4) edge[color=blue, opacity=0.4, bend left=20] (B5); 
    \path (B5) edge[color=blue, opacity=0.4, bend left=20] (B6); 
    \path (B7) edge[color=blue, opacity=0.4, bend left=20] (B8); 
    \path (B8) edge[color=blue, opacity=0.4, bend left=20] (B9); 
    \path (C1) edge[color=blue, opacity=0.4, bend left=20] (C2); 
    \path (C2) edge[color=blue, opacity=0.4, bend left=20] (C3); 
    \path (C4) edge[color=blue, opacity=0.4, bend left=20] (C5); 
    \path (C5) edge[color=blue, opacity=0.4, bend left=20] (C6); 
    \path (C7) edge[color=blue, opacity=0.4, bend left=20] (C8); 
    \path (C8) edge[color=blue, opacity=0.4, bend left=20] (C9); 
    \path (D1) edge[color=blue, opacity=0.4, bend left=20] (D2); 
    \path (D2) edge[color=blue, opacity=0.4, bend left=20] (D3); 
    \path (D4) edge[color=blue, opacity=0.4, bend left=20] (D5); 
    \path (D5) edge[color=blue, opacity=0.4, bend left=20] (D6); 
    \path (D7) edge[color=blue, opacity=0.4, bend left=20] (D8); 
    \path (D8) edge[color=blue, opacity=0.4, bend left=20] (D9); 
    \path (E1) edge[color=blue, opacity=0.4, bend left=20] (E2); 
    \path (E2) edge[color=blue, opacity=0.4, bend left=20] (E3); 
    \path (E4) edge[color=blue, opacity=0.4, bend left=20] (E5); 
    \path (E5) edge[color=blue, opacity=0.4, bend left=20] (E6); 
    \path (E7) edge[color=blue, opacity=0.4, bend left=20] (E8); 
    \path (E8) edge[color=blue, opacity=0.4, bend left=20] (E9);
    \path (A1) edge[color=green,bend left=20] (A5); 
    \path (A2) edge[color=green,bend left=20] (A6); 
    \path (A4) edge[color=green,bend left=20] (A8); 
    \path (A5) edge[color=green,bend left=20] (A9); 
    \path (B1) edge[color=green,bend left=20] (B5); 
    \path (B2) edge[color=green,bend left=20] (B6); 
    \path (B4) edge[color=green,bend left=20] (B8); 
    \path (B5) edge[color=green,bend left=20] (B9); 
    \path (C1) edge[color=green,bend left=20] (C5); 
    \path (C2) edge[color=green,bend left=20] (C6); 
    \path (C4) edge[color=green,bend left=20] (C8); 
    \path (C5) edge[color=green,bend left=20] (C9); 
    \path (D1) edge[color=green,bend left=20] (D5); 
    \path (D2) edge[color=green,bend left=20] (D6); 
    \path (D4) edge[color=green,bend left=20] (D8); 
    \path (D5) edge[color=green,bend left=20] (D9); 
    \path (E1) edge[color=green,bend left=20] (E5); 
    \path (E2) edge[color=green,bend left=20] (E6); 
    \path (E4) edge[color=green,bend left=20] (E8); 
    \path (E5) edge[color=green,bend left=20] (E9); 
    \path (S1) edge[color=green,bend left=20] (S5); 
    \path (S2) edge[color=green,bend left=20] (S6); 
    \path (S4) edge[color=green,bend left=20] (S8); 
    \path (S5) edge[color=green,bend left=20] (S9);
    \path (A1) edge[color=red, opacity=0.4, bend left=20] (A4); 
    \path (A4) edge[color=red, opacity=0.4, bend left=20] (A7);
    \path (A2) edge[color=red, opacity=0.4, bend left=20] (A5); 
    \path (A5) edge[color=red, opacity=0.4, bend left=20] (A8); 
    \path (A3) edge[color=red, opacity=0.4, bend left=20] (A6); 
    \path (A6) edge[color=red, opacity=0.4, bend left=20] (A9); 
    \path (B1) edge[color=red, opacity=0.4, bend left=20] (B4); 
    \path (B4) edge[color=red, opacity=0.4, bend left=20] (B7);
    \path (B2) edge[color=red, opacity=0.4, bend left=20] (B5); 
    \path (B5) edge[color=red, opacity=0.4, bend left=20] (B8); 
    \path (B3) edge[color=red, opacity=0.4, bend left=20] (B6); 
    \path (B6) edge[color=red, opacity=0.4, bend left=20] (B9); 
    \path (C1) edge[color=red, opacity=0.4, bend left=20] (C4); 
    \path (C4) edge[color=red, opacity=0.4, bend left=20] (C7);
    \path (C2) edge[color=red, opacity=0.4, bend left=20] (C5); 
    \path (C5) edge[color=red, opacity=0.4, bend left=20] (C8); 
    \path (C3) edge[color=red, opacity=0.4, bend left=20] (C6); 
    \path (C6) edge[color=red, opacity=0.4, bend left=20] (C9); 
    \path (D1) edge[color=red, opacity=0.4, bend left=20] (D4); 
    \path (D4) edge[color=red, opacity=0.4, bend left=20] (D7);
    \path (D2) edge[color=red, opacity=0.4, bend left=20] (D5); 
    \path (D5) edge[color=red, opacity=0.4, bend left=20] (D8); 
    \path (D3) edge[color=red, opacity=0.4, bend left=20] (D6); 
    \path (D6) edge[color=red, opacity=0.4, bend left=20] (D9); 
    \path (E1) edge[color=red, opacity=0.4, bend left=20] (E4); 
    \path (E4) edge[color=red, opacity=0.4, bend left=20] (E7);
    \path (E2) edge[color=red, opacity=0.4, bend left=20] (E5); 
    \path (E5) edge[color=red, opacity=0.4, bend left=20] (E8); 
    \path (E3) edge[color=red, opacity=0.4, bend left=20] (E6); 
    \path (E6) edge[color=red, opacity=0.4, bend left=20] (E9); 
    \path (S1) edge[color=red, opacity=0.4, bend left=20] (S4); 
    \path (S4) edge[color=red, opacity=0.4, bend left=20] (S7);
    \path (S2) edge[color=red, opacity=0.4, bend left=20] (S5); 
    \path (S5) edge[color=red, opacity=0.4, bend left=20] (S8); 
    \path (S3) edge[color=red, opacity=0.4, bend left=20] (S6); 
    \path (S6) edge[color=red, opacity=0.4, bend left=20] (S9); 
\end{tikzpicture}
    \caption{(a) The star $S_5$, (b) the graph $H_3$, and (c) their Cartesian product $S_5\square H_3$.
	The four edge colours illustrate a queue layout for $S_5\square H_3$.}
	\label{fig:SnHn}
\end{figure}
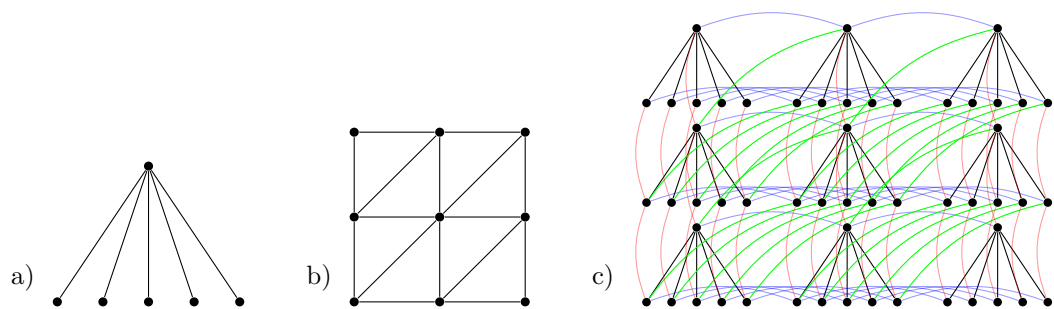

Recall that $S_n$ is the star with $n$ leaves, and that $G_1\square G_2$ denotes the Cartesian product of two graphs $G_1$ and~$G_2$.
Dujmovi\'c et al.~\cite{dujmovic2021stack} showed that, for all integers $a,n>0$ and the Cartesian product $G=S_a\square H_n$,
we have $\qn(G)\leq4$.
In fact, they noted that every $H_n$ admits a so-called {\em strict} $3$-queue layout, 
which ``adds up'' with a trivial $1$-queue layout of $S_a$ over Cartesian product by Wood~\cite{wood2005queue}.
Their main result reads:

\begin{theorem}[Dujmovi\'c et al.~\cite{dujmovic2021stack}]\label{thm:main}
For every integer $s$, and for $a,n>0$ which are sufficiently large with respect to~$s$,
the Cartesian product $G:= S_a \square H_n$ is of stack number at least~$s$.
\end{theorem}

Our contribution is to give a very short simplified proof of \Cref{thm:main} (based in parts on the ideas from \cite{dujmovic2021stack},
but also eliminating some rather long fragments of the former proof).

\section{Proof of \Cref{thm:main}}
%%%%%%%%%%%%%%%%%%%%%%%%%%%%%%%%%%%%%%%%%%%%%%%%%%%%%%%%%%%%%%%%%%%%%%%
% \section{A queue number bound on the stack number}
\label{chapter_proof}

We will use some classical results, the first two of which are truly folklore.
\begin{proposition}[Ramsey~\cite{ramsey}]\label{ramsey}
For all integers $r,s>0$ there exists $R=R(r,s)$ such that for any assignment of two colours read and blue to the edges of 
the complete graph $K_R$, there is a red clique on $r$ vertices or a blue clique on $s$ vertices in it.
\end{proposition}
\begin{proposition}[Erd\H{o}s--Szekeres~\cite{erdos}]\label{erdos-szekeres}
For given integers $r,s>0$, any sequence of distinct elements of a linearly ordered set of length more than $r\cdot s$ 
contains an increasing subsequence of length $s+1$ or a decreasing subsequence of length~$r+1$.
\end{proposition}
\begin{proposition}[Gale~\cite{hex}]\label{hex_lemma}
Consider a dual hexagonal grid $H_n$ as above. For any assignment of two colours to the vertices of $H_n$,
there exists a monochromatic path on $n$ vertices.
\end{proposition}

Consider for the rest any fixed stack layout of the graph $G$ of \Cref{thm:main}, with the linear order~$\prec$ on the vertex set~$V(G)$.
Recall that $V(G)=\{(u,p): u\in V(S_a), p\in V(H_n)\}$.

\begin{lemma}
Let $L$ be the set of leaves of $S_a$, and let $b=a^{1/m}$ where $m=2^{n^2-1}$.
There is a subsequence $(u_1,\ldots,u_b)$ in the set $L$ of length $b$ such that for each $p \in V(H_n)$, 
either $(u_1, p) \prec (u_2, p) \prec \ldots \prec  (u_b, p)$, or $(u_1, p) \succ (u_2, p) \succ \ldots \succ  (u_b, p)$.
\label{erdos_lemma}
\end{lemma}
\begin{proof}
Let $V(H_n)=\{p_1,\ldots,p_{n^2}\}$ be the vertices of $H_n$.
Let $L=\{u[1],u[2],\ldots,u[a]\}$ where we write $u[i]$ instead of traditional index form $u_i$ to avoid nested indices.
Start with $a_1=a$ and the permutation $\sigma_1=\big(i_{1,1},\ldots,i_{1,a_1}\big)$ of $\{1,\ldots,a\}$
such that $(u[i_{1,1}], p_1) \prec\ldots\prec (u[i_{1,a_1}], p_1)$.
By iterating \Cref{erdos-szekeres}, for each $j=2,\ldots,n^2$, the sequence $\sigma_{j-1}$
contains a subsequence $\sigma_j = (i_{j,1}, \ldots, i_{j,a_j})$ such that $a_j\geq \sqrt{a_{j-1}}$,
and $(u[i_{j,1}], p_j) \prec\ldots \prec (u[i_{j,a_j}],p_j)$ or $(u[i_{j,1}], p_j) \succ\ldots \succ (u[i_{j,a_j}],p_j)$.
By simple calculus, we get $a_{n^2}\geq a_1^{1/m}=b$ which is the desired outcome.
\end{proof}

Let $S_b\subseteq S_a$ be the (specific) substar of $S_a$ defined by the subset of leaves $\{u_1,\ldots,u_b\}$ (of \Cref{erdos_lemma}).
Colour every vertex $p \in V(H_n)$ red if $(u_1, p) \prec \ldots \prec (u_b, p)$, and colour $p$ blue otherwise.
From this and \Cref{hex_lemma}, we immediately obtain:
\begin{corollary}\label{cor:sameord}
There is a subgraph $Q\subseteq H_n$, being a path on $n$ vertices, such that, without loss of generality, 
$(u_1, q) \prec \ldots \prec (u_b, q)$ holds for every vertex $q\in V(Q)$.
\qed\end{corollary}

Define $X\subseteq G$ to be the subgraph induced on the vertex set $V(S_b)\times V(Q)$, i.e., $X=S_b\square Q$,
and denote by $\ca R$ the set of paths $R_i\subseteq X$ induced on $\{u_i\}\times V(Q)$ for $i=1,\ldots,b$.
We extend $\prec$ to a partial order on $\ca R$ as follows; for $R_i,R_j\in\ca R$, we have $R_i\prec R_j$,
if and only if $u\prec w$ for all $u\in V(R_i)$ and $w\in V(R_j)$.
We say that $R_i$ and $R_j$ are {\em$\prec$-separated} if $R_i\prec R_j$ or $R_i\succ R_j$, and that
$R_i$ and $R_j$ are {\em$\prec$-crossing} if there exist edges $e\in E(R_i)$ and $f\in E(R_j)$ such that $e,f$ $\prec$-cross.
The following result is simple but crucial:

\begin{lemma}\label{lem:crossepar}
    Every two distinct paths $R_i, R_j \in \mathcal{R}$ are either $\prec$-crossing, or $\prec$-separated.
\end{lemma}
\begin{proof}
    Assume the contrary; up to symmetry meaning that all edges of $R_i$ are nested in some edge $e_2 = \{(u_j,q),(u_j,q')\}\in E(R_j)$.
Then, in particular, $e_1 = \{(u_i,q),(u_i,q')\}\in E(R_i)$ is nested in $e_2$, and so $(u_j, q) \prec (u_i, q)$ and $(u_j, q') \succ (u_i, q')$.
This contradicts \Cref{cor:sameord}.
\end{proof}

\begin{corollary}\label{cor:crossepar}
For all integers $c,d$ and $n$, and for $b=|\ca R|$ sufficiently large with respect to $c,d$, we have that $\ca R$ contains
at least $c$ pairwise $\prec$-separated or $d$ pairwise $\prec$-crossing paths.
\end{corollary}
\begin{proof}
Imagine a pair of paths $\{R_i,R_j\}\subseteq\ca R$ coloured red if $R_i,R_j$ are $\prec$-crossing, and blue if they are $\prec$-separated.
With respect to \Cref{lem:crossepar}, we apply \Cref{ramsey} with~$b\geq R(c,d)$.
\end{proof}

We finish as follows.
\begin{proof}[Proof of~\Cref{thm:main}]
Respecting the above definition of the set of paths $\ca R$ in $G$, we branch into the two cases determined by \Cref{cor:crossepar}.
\smallskip
\begin{description}
\item[Case~I.]
 There are $c$ pairwise $\prec$-separated paths in $\mathcal{R}$. 

Without loss of generality, let these paths be $R_1 \prec \ldots \prec R_c$.
For the root $t$ of $S_b$, label the $n$ vertices of the set $\{t\}\times V(Q)\subseteq V(X)$ by $t_1\prec\ldots\prec t_n$. 
There are two subcases.
\begin{itemize}
     \item  $R_{\lfloor c/2 \rfloor} \prec t_{\lceil n/2 \rceil}$. 
For each $i=1,\ldots,\min(\lfloor c/2\rfloor,\lceil n/2\rceil)$, pick an edge of $X$ from $t_{\lceil n/2\rceil+i-1}$ to $V(R_i)$ 
(which exist since $R_i$ hits every copy of $S_b$ in~$X$ by the definition).
We have got $\min(\lfloor c/2\rfloor,\lceil n/2\rceil)$ edges in $X$ that pairwise $\prec$-cross, as in \Cref{fig:separat}.
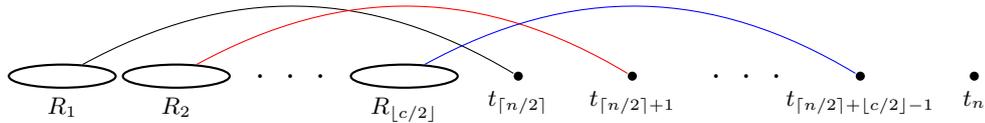
\begin{figure}[h!]
    \centering
\begin{tikzpicture}[every node/.style={thick}]\small
    \node [style=ellipse, minimum width=40pt, label=below:$R_1$, draw](A) at (0,0) {};
    \node [style=ellipse, minimum width=40pt, label=below:$R_2$, draw](B) at (1.5,0) {};
    \node (C) [ellipse, minimum width=100pt, color=white, text = black]at (3,0) {\LARGE . . .};
    \node [ellipse, minimum width=40pt, label=below:$R_{\lfloor c/2 \rfloor}$, draw](D) at (4.5,0) {};
    \node [style={circle,fill,inner sep=0pt, minimum size=1.2mm}, label=below:$t_{\lceil n/2 \rceil}$](A') at (6,0) { };
    \node [style={circle,fill,inner sep=0pt, minimum size=1.2mm}, label=below:$t_{\lceil n/2 \rceil + 1}$](B') at (7.5,0) { };
    \node [ style={circle,fill,inner sep=0pt, minimum size=1.2mm}, color = white, text = black](C') at (9,0) {\LARGE . . .};
    \node [style={circle,fill,inner sep=0pt, minimum size=1.2mm}, label=below:$t_{\lceil n/2 \rceil+\lfloor c/2 \rfloor-1}$](D') at (10.5,0) { };
    \node [style={circle,fill,inner sep=0pt, minimum size=1.2mm}, label=below:$t_n$](D'') at (12,0) { };
    \path (A) edge[bend left=30] (A'); 
    \path (B) edge[bend left=30, color=red] (B'); 
    \path (D) edge[bend left=30, color=blue](D'); 
\end{tikzpicture}
    \caption{Case~I, where $R_{\lfloor c/2 \rfloor} \prec t_{\lceil n/2 \rceil}$ and $\lceil n/2\rceil>\lfloor c/2\rfloor$.}
\label{fig:separat}
\end{figure} 
     \item  $t_{\lceil n/2 \rceil} \prec R_{\lfloor c/2 \rfloor + 1}$ (note that $t_{\lceil n/2\rceil}$ may be ``$\prec$-nested'' in  $R_{\lfloor c/2 \rfloor}$). 
This is symmetric to the previous, and we get $\min(\lceil c/2 \rceil, \lceil n/2 \rceil)$ pairwise $\prec$-crossing edges in $X$ 
between vertices of $R_{\lfloor c/2 \rfloor + 1},\ldots, R_c$ and $t_1,\ldots, t_{\lceil n/2 \rceil}$.
% \begin{figure}[h!]
%     \centering
%     \trimbox{0pt, 10pt, 0pt, 15pt}{
%     \input{pictures/separated_b}
%     }
%     \caption{Case \ref{case_1}, where $s_{\lceil n/2 \rceil} \prec R_{\lceil c/2 \rceil + 1}$}
% \end{figure} 
\end{itemize}
\smallskip
\item[Case~II.]
    There are $d$ pairwise $\prec$-crossing paths in $\mathcal{R}$.

Pick any path $R_0$ out of these $d$ paths. In $Z:=\bigcup_{R\in\ca R, R\not=R_0}E(R)$ there are at least $d-1$ edges 
which $\prec$-cross some edge of $R_0$, and so at least $(d-1)/n$ of them cross the same edge $e\in E(R_0)$.
Having $e=u_1u_2$, $u_1\prec u_2$, and $f=v_1v_2\in E(X)$ such that $e$ and $f$ $\prec$-cross, we say that $v_1$ is the {\em inside}
vertex of $f$ if $u_1\prec v_1\prec u_2$, and then $v_2$ is the {\em outside} vertex.
By the pigeonhole principle, there is a set $Z'\subseteq Z$ of $d'=|Z'|\geq(d-1)/n^2$ edges $\prec$-crossing $e$ such that their inside vertices
belong to the same copy of the substar~$S_b$~in~$X$.

The outside vertices of the edges of $Z'$ belong to at most two other copies of $S_b$ in~$X$ (determined by a neighbourhood in the path $Q$),
and each is before or after $e$ in~$\prec$.
By the pigeonhole principle again, and without loss of generality, there is a set $Z''\subseteq Z'$ of size $|Z''|\geq\frac12\cdot\frac12d'=d'/4$,
such that also the outside vertices of the edges of $Z''$ belong to the same copy of $S_b$ in~$X$, and they all lie after $e$ in~$\prec$.
See \Cref{fig:crossi}.
Moreover, by \Cref{cor:sameord} (the ordering claimed therein), the edges in $Z''$ must pairwise $\prec$-cross.

\begin{figure}[th]
    \centering
\begin{tikzpicture}[scale=0.85]\small
    \node (w) [style={circle,fill,inner sep=0pt, minimum size=0.12cm}, label=below:$u_1$] at (0, 0) { };
    \node (ww)[style={circle,fill,inner sep=0pt, minimum size=0.12cm}, label=below:$u_2$] at (6, 0) { };
    \node (x)[style={opacity=0, label}, label=below:$e$] at (0.5, 1) {};
    \path (w) edge[thick, bend left = 40, color=red] (ww);
    
    \node (S1)[style={circle,fill,inner sep=0pt, minimum size=0.12cm}, label=left:$t_1$] at (1.5,1.7) { };
    \node (A1)[style={circle,fill,inner sep=0pt, minimum size=0.12cm}] at (2,0) { };
    \node (B1)[style={circle,fill,inner sep=0pt, minimum size=0.12cm}] at (2.5,0) { };
    \node (C1)[style={circle,fill,inner sep=0pt, minimum size=0.12cm}] at (3,0) { };
    \node (D1)[style={circle,fill,inner sep=0pt, minimum size=0.12cm}] at (3.5,0) { };
    \node (E1)[style={circle,fill,inner sep=0pt, minimum size=0.12cm}] at (4,0) { };    
    \node (x)[style={opacity=0, label}, label=below:$Z''$] at (6.5, 1) {};
    \path (S1) edge[opacity=0.5] (A1); 
    \path (S1) edge[opacity=0.5] (B1); 
    \path (S1) edge[opacity=0.5] (C1);
    \path (S1) edge[opacity=0.5] (D1);
    \path (S1) edge[opacity=0.5] (E1);

    \node (S2)[style={circle,fill,inner sep=0pt, minimum size=0.12cm}, label=right:$t_2$] at (10.5,1.7) { };
    \draw[opacity=0.5] (S1)--(S2);
    \node (A2)[style={circle,fill,inner sep=0pt, minimum size=0.12cm}] at (8,0) { };
    \node (B2)[style={circle,fill,inner sep=0pt, minimum size=0.12cm}] at (8.5,0) { };
    \node (C2)[style={circle,fill,inner sep=0pt, minimum size=0.12cm}] at (9,0) { };
    \node (D2)[style={circle,fill,inner sep=0pt, minimum size=0.12cm}] at (9.5,0) { };
    \node (E2)[style={circle,fill,inner sep=0pt, minimum size=0.12cm}] at (10,0) { };    
    \path (S2) edge[opacity=0.5] (A2); 
    \path (S2) edge[opacity=0.5] (B2); 
    \path (S2) edge[opacity=0.5] (C2);
    \path (S2) edge[opacity=0.5] (D2);
    \path (S2) edge[opacity=0.5] (E2);

    \path (A1) edge[color=blue, bend left = 40, opacity=60] (A2);
    \path (B1) edge[color=blue, bend left = 40, opacity=60] (B2);
    \path (C1) edge[color=blue, bend left = 40, opacity=60] (C2);
    \path (D1) edge[color=blue, bend left = 40, opacity=60] (D2);
    \path (E1) edge[color=blue, bend left = 40, opacity=60] (E2);
\end{tikzpicture}
    \caption{Case II, with emphasized edge $e$, vertices $t_1,t_2$ being two copies of the root of~$S_b$,
	and the selected pairwise-crossing edges of $Z''$ in blue.}
    \label{fig:crossi}
\end{figure}
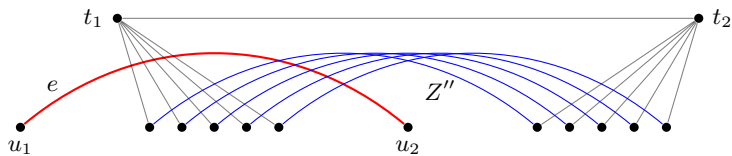
\end{description}

To finish the proof, we set $n=2s$ and $a = R(2s, 4n^2s+1)^{m}$ where $m=2^{n^2-1}$.
Then in \Cref{erdos_lemma} we get $b=R(2s, 4n^2s+1)$, and in \Cref{cor:crossepar} we have $c=2s$ and $d=4n^2s+1$.
In Case I, we then obtain at least $\min(\lfloor c/2\rfloor,\lceil n/2\rceil)=s$ edges of $X\subseteq G$ that pairwise $\prec$-cross.
In Case II, it is at least $d'/4=(d-1)/(4n^2)=s$ such pairwise $\prec$-crossing edges, too.
Edges that pairwise $\prec$-cross obviously must receive distinct colours.
A valid stack layout based on $\prec$ hence needs at least $s$ colours, and since $\prec$ has been arbitrary for the graph~$G$, 
we finally conclude that $\sn(G)\geq s$.
\end{proof}

\section{Conclusion}
%%%%%%%%%%%%%%%%%%%%%%%%%%%%%%%%%%%%%%%%%%%%%%%%%%%%%%%%%%%%%%%%%%%%%%%

We have provided a short elementary proof of \Cref{thm:main}.
Although the original proof in \cite{dujmovic2021stack} is not very long or difficult, by carefully rearranging the arguments
we have succeeded in eliminating some technical steps of the proof in \cite{dujmovic2021stack}
and, in particular, resolved the case of pairwise crossing paths in a direct short way.
Briefly explaining, our proof skips initial technical parts of \cite{dujmovic2021stack} preceding the use of
\Cref{erdos-szekeres} (Erd\H{o}s--Szekeres) and readily applies \Cref{erdos-szekeres} and \Cref{hex_lemma} in a way similar to \cite{dujmovic2021stack}. Then, instead of proving the existence of a set of pairwise separated paths in $G$, we conclude by \Cref{ramsey} (Ramsey) in which both outcomes straightforwardly lead to a large set of pairwise crossing edges,
thus avoiding other technical steps needed in \cite{dujmovic2021stack} mainly at the end of the arguments.

% The presented proof is based on the Bachelor's thesis of the second author \cite{Straka2023thesis,straka2023stack}.

% We have provided a short self-contained proof of \Cref{thm:twwplanar}.
% While the proved bound is not the best currently possible, % cf.~\cite{DBLP:journals/corr/abs-2210-08620},
% the proof given here is way much simpler than those in \cite{DBLP:journals/corr/abs-2205.05378,DBLP:journals/corr/abs-2210-08620}.
% While sacrificing a bit of simplicity of the given proof, we can also give a better upper bound
% of~$9$ (thus matching \cite{DBLP:journals/corr/abs-2205.05378}), but
% we are so far not sure whether a similarly simplified proof can be given for the upper bound of $8$ as in~\cite{DBLP:journals/corr/abs-2210-08620}.

\bibliography{stackqueue}

\end{document}